\documentclass[oneside,11pt]{amsart}
\usepackage{amsmath}
\usepackage{amssymb}
\usepackage{graphicx}
\usepackage{caption}
\usepackage{subcaption}
\usepackage{pinlabel}
\usepackage{tikz}

\newtheorem{thm}{Theorem}[section]
\newtheorem{prop}[thm]{Proposition}
\newtheorem{lem}[thm]{Lemma}
\newtheorem{cor}[thm]{Corollary}
\newtheorem{conj}[thm]{Conjecture}

\theoremstyle{definition}
\newtheorem{defn}[thm]{Definition}
\newtheorem{rem}[thm]{Remark}

\newtheorem{exmp}[thm]{Example}

\newcommand{\abs}[1]{\lvert{#1}\rvert}

\renewcommand{\bar}[1]{\overline{#1}}

\newcommand{\set}[2]{\{\,{#1} \mid {#2} \,\}}
\newcommand{\bigset}[2]{ \bigl\{ \, {#1} \bigm| {#2} \, \bigr\} }
\renewcommand{\emptyset}{\varnothing}

\newcommand{\field}[1]{\mathbb{#1}}
\newcommand{\Z}{\field{Z}}

\DeclareMathOperator{\CAT}{CAT}

\DeclareMathOperator{\diam}{diam}

\begin{document}

\title[On distortion of normal subgroups]{On distortion of normal subgroups}

\author{Hung Cong Tran}
\address{Department of Mathematics\\
 The University of Georgia\\
1023 D. W. Brooks Drive\\
Athens, GA 30605\\
USA}
\email{hungtran280687@gmail.com}

\date{\today}

\begin{abstract}
We examine distortion of finitely generated normal subgroups. We show a connection between subgroup distortion and group divergence. We suggest a method computing the distortion of normal subgroups by decomposing the whole group into smaller subgroups. We apply our work to compute the distortion of normal subgroups of graph of groups and normal subgroups of right-angled Artin groups that induce infinite cyclic quotient groups. We construct normal subgroups of $\CAT(0)$ groups introduced by Macura and introduce a collection of normal subgroups of right-angled Artin groups. These groups provide a rich source to study the connection between subgroup distortion and group divergence on $\CAT(0)$ groups.
\end{abstract}

\maketitle

\section{Introduction}
Subgroup distortion is a famous tool to study the geometric connection between a finitely generated group and its finitely generated subgroups. A metric on a subgroup can be inherited from the metric of the whole group. Also a finitely generated subgroup can be a metric space itself. Therefore, the subgroup distortion notion was introduced to measure the difference between the two metric structures of a finitely generated subgroup. In this article, we only focus on distortion of finitely generated normal subgroups. 

The divergence is a quasi-isometry invariant which arose in the study of non-positively curved manifolds and metric spaces. Roughly speaking, the divergence of a connected metric space measures the complement distance of a pair of points in a sphere as a function of the radius. For example, the divergence of the Euclidean plane is linear and the divergence of the hyperbolic plane is exponential. The concept of divergence on finitely generated groups is defined via their Cayley graphs.

Connection between group divergence and distortion of normal subgroups was first shown by Gersten \cite{MR1254309}. He proved the distortion of a finitely generated normal subgroup $H$ in a finitely generated $G$ is bounded below by the divergence of $G$ when the quotient group $G/H$ is an infinite cyclic group. We generalize the result of Gersten by the following theorem:

\begin{thm}
\label{ip1}
Let $G$ be a finitely generated group and $H$ a finitely generated infinite index infinite normal subgroup of $G$. Then the divergence of $G$ is dominated by the subgroup distortion of $H$ in $G$.
\end{thm}

By the above theorem, we can use the divergence of the whole group as a lower bound for distortion of any finitely generated normal subgroup. Also, we can use subgroup distortion as a tool to compute the upper bound for group divergence. In this paper, we revisit Macura's examples on $\CAT(0)$ groups with polynomial divergence (see \cite{MR3032700}) and show an alternative way using subgroup distortion to compute the upper bound of the divergence of these groups (see Remark \ref{remark1}). 

In \cite{Sisto}, Sisto proved that the divergence of a one-ended nontrivial relatively hyperbolic group is at least exponential. We are not sure if the study of the distortion of finitely generated normal subgroups in finitely generated relatively hyperbolic groups is recorded in the literature. However, the following corollary is a direct application of Theorem \ref{ip1}.

\begin{cor}
Let $G$ be a one-ended relatively hyperbolic group with respect to some collection of proper subgroups and $H$ a finitely generated infinite index infinite normal subgroup of $G$. Then the distortion of $H$ in $G$ is at least exponential.
\end{cor}

We study the upper bound for the distortion of a normal subgroup by decomposing the whole group into smaller subgroups and computing the distortion of the normal subgroup on each of these subgroups.

\begin{thm}
\label{ip2}
Let $G$ be a finitely generated group and $H$ normal subgroup of $G$ with the canonical projection $p:G \rightarrow G{/}H$. Suppose there is a finite collection $\mathcal{A}$ of finitely generated subgroups of $G$ that satisfies the following conditions:
\begin{enumerate}
\item The union of all subgroups in $\mathcal{A}$ generates $G$ and the subgroup $A\cap H$ is finitely generated for each subgroup $A$ in $\mathcal{A}$.
\item For each pair $A$ and $A'$ in $\mathcal{A}$, there is a sequence of subgroups in $\mathcal{A}$ $A=A_1, A_2, \cdots, A_n=A'$ such that $p(A_i\cap A_{i+1})$ is a finite index subgroup of $G/H$ for each $i$ in $\{1,2, \cdots, n-1\}$.
\end{enumerate}
Then, the subgroup $H$ is finitely generated and the distortion $Dist_{G}^H$ is dominated by the function $f(n)=n \max \set{Dist_{A}^{A\cap H}(n)}{A \in \mathcal{A}}$.
\end{thm}

The above theorem can be used to compute the distortion of normal subgroups of a simply connected finite graph of groups (see the following corollary).

\begin{cor}
Let $G$ be the fundamental group of a simply connected finite graph of finitely generated groups and $H$ a finitely generated normal subgroup of $G$ not contained in any edge subgroup. We assume that the intersection of $H$ with any vertex subgroup is finitely generated. Then, the distortion $Dist_{G}^H$ is dominated by the function $f(n)=n \max \set{Dist_{A}^{A\cap H}(n)}{A \text{ is a vertex subgroup}}$.
\end{cor}

We also use Theorem \ref{ip2} to compute the distortion of normal subgroups of right-angled Artin groups that induce infinite cyclic quotient groups. 

\begin{cor}
\label{corag}
Let $\Gamma$ be a simplicial connected graph with at least two vertices. Let $p$ be an arbitrary group homomorphism from the right-angled Artin group $A_\Gamma$ to $\Z$ and $N$ the kernel of the map $p$. Let $\Gamma'$ be the induced subgraph of $\Gamma$ with the vertices which are mapped to nonzero numbers in $\Z$ by $p$. If the graph $\Gamma'$ is connected and the star of $\Gamma'$ in $\Gamma$ contains all vertices of $\Gamma$, then the subgroup $N$ is finitely generated and the distortion of $N$ in $A_\Gamma$ is at most quadratic. Moreover, if each vertex of $\Gamma$ is mapped to a nonzero number in $\Z$, then the distortion of the $N$ in $A_\Gamma$ is linear when $\Gamma$ is a join and the distortion is quadratic otherwise. 
\end{cor}
In the above corollary, the subgroup $N$ is the Bestvina-Brady subgroup of $A_\Gamma$ if $p$ maps each vertex of $\Gamma$ to 1. We remark that the distortion of Bestvina-Brady subgroups were also computed by the author in \cite{Tran}. The kernel of the group homomorphism $p:A_\Gamma\rightarrow \Z$ was also studied by Papadima-Suciu in \cite{MR2466422}. The divergence of right-angled Artin groups was studied by Behrstock--Charney in \cite{MR2874959} and we used the divergence of the right-angled Artin group $A_\Gamma$ to prove for the lower bound of the distortion of $N$. However, the upper bound of the distortion of $N$ is calculated by using Theorem \ref{ip2}. Moreover, the upper bound of the distortion of $N$ provides an alternative method to compute the upper bound of the divergence of the right-angled Artin group $A_{\Gamma}$.

The notion of algebraic thickness of groups was introduced in \cite{MR2501302} to study the geometry of non-relatively hyperbolic groups. The collection of algebraically thick groups includes many types of groups such as mapping class groups of surfaces $S$ with $3\times\text{genus}(S)+\# \text{punctures} \geq 5$, $Aut(F_n)$ and $Out(F_n)$ for $n\geq 3$, non-relatively hyperbolic Coxeter groups, various Artin groups, and others. Behrstock-Dru{\c{t}}u \cite{MR3421592} used algebraic thickness of a group to compute the upper bound of its divergence. Since right-angled Artin groups with connected defining graphs are examples algebraically thick groups, Corollary \ref{corag} gives us a hint to use Theorem \ref{ip2} to study distortion of normal subgroups of algebraically thick groups. 

More precisely, the algebraic thick order of group $G$ is defined inductively through a decomposition of a finite index subgroup $G_1$ into a finite collection $\mathcal{H}$ of subgroups with lower algebraic thick order that satisfies some certain conditions (see \cite{MR2501302} for the precise definition). Therefore, we can use Theorem \ref{ip2} to compute the upper bound of distortion of a normal subgroup $H$ in $G$ if the subgroup $H\cap G_1$ and the collection $\mathcal{H}$ satisfy all hypothesis in the theorem and the distortion of $H$ on each subgroup in $\mathcal{H}$ is well-understood. Also divergence functions of many thick groups were known. Therefore, we can apply Theorem~\ref{ip1} to study the distortion normal subgroups of these groups.

Macura \cite{MR3032700} introduced a class of $\CAT(0)$ groups (i.e. \emph{groups which act properly and cocompactly on some $\CAT(0)$ spaces}) to study group divergence. She showed that for each positive integer $d$ there is a $\CAT(0)$ groups with polynomial of degree $d$ divergence function. We revisit Macura's examples and construct normal subgroups for each example with distortion function as the same as the divergence function of the whole group. 

\begin{thm}
For each integer $d\geq 2$ let $G_d=\langle a_0, a_1,\cdots, a_d |a_0a_1=a_1a_0, a^{-1}_ia_0a_i=a_{i-1}, \text{for $2\leq i\leq d$}\rangle$. Then $G_d$ acts properly and cocompactly on some $CAT(0)$ space with polynomial divergence function of degree $d$ and we can find a normal finitely generated free subgroup $F_d$ inside such that the quotient $G_d/F_d$ is an infinite cyclic subgroup and the distortion of $F_d$ is also a polynomial function of degree $d$.
\end{thm}

We remark that each group $G_d$ in the above theorem was constructed by Macura \cite{MR3032700} and the property that each group $G_d$ acts properly and cocompactly on some $\CAT(0)$ space with polynomial divergence function of degree $d$ was also proved by her. We only construct the free normal subgroup $F_d$ in each $G_d$ with polynomial distortion function of degree $d$. We also remark that the existence of subgroups of $\CAT(0)$ groups with polynomial distortion was known before (see \cite{MR2252898} for example). However, the purpose of the above theorem is to emphasize on the connection between distortion of normal free subgroups and group divergence functions on $\CAT(0)$ groups. We also hope that the theorem can shed a light for the positive answer to the following conjecture raised by Gersten \cite{MR1254309}:

\begin{conj}\cite{MR1254309}
Let $F$ be a finitely generated normal free subgroup in a group $G$ such that the quotient $G/F$ is an infinite cyclic group. Then the distortion of $F$ in $G$ is equivalent to the divergence of $G$.
\end{conj} 

We introduce a collection of normal subgroups of right-angled Artin groups, called generalized Bestvina-Brady subgroups (see Section~\ref{GBBS} for the precise definition). By using the work of Meier-VanWyk \cite{MR1337468}, we give a necessary and sufficient condition for these groups to be finitely generated.

\begin{thm}
Let $\Gamma$ be a finite simplicial graph. Let $H_\Phi$ be a generalized Bestvina-Brady subgroup of $A_{\Gamma}$. Then $H_\Phi$ is finitely generated iff each basis subgraph of $\Gamma$ with respect to $\Phi$ is connected and dominating. 
\end{thm}

 We compute the distortion of groups of this collection with some additional hypothesis by mainly using Theorem \ref{ip1} and Theorem \ref{ip2}.

\begin{thm}
Let $\Gamma$ be a simplicial connected graph. Let $H_\Phi$ be a strong generalised Bestvina-Brady subgroup of $A_\Gamma$ with basis subgraphs $\{\Gamma_i\}_{1\leq i \leq d}$. Then, the distortion of $H_\Phi$ in $A_\Gamma$ is quadratic if there is a non-empty subset $I$ of $\{1,2, \cdots, d\}$ such that the subgraph $\Gamma'$ generated by $\{\Gamma_i\}_{i\in I}$ is not a join. Moreover, the distortion of $H_\Phi$ in $A_\Gamma$ is linear if one of the following conditions holds:
\begin{enumerate}
\item The generalised Bestvina-Brady subgroup $H_\Phi$ is a special generalised Bestvina-Brady subgroup.
\item Each basis subgraphs in $\{\Gamma_i\}_{1\leq i \leq d}$ is a join and each pair of basis subgraphs in $\{\Gamma_i\}_{1\leq i \leq d}$ commutes. 
\end{enumerate}
\end{thm}

The above theorem provides a rich source of examples to study the connection between group divergence and subgroup distortion. We also hope that the collection of generalized Bestvina-Brady subgroups also provides a good source of examples to study normal subgroups of right-angled Artin groups.

\subsection*{Acknowledgments}
I would like to thank Prof.~Christopher Hruska, Ignat Soroko, and Hoang Thanh Nguyen for their helpful comments and suggestions. I also want to thank Prof.~Jason Behrstock for his very useful conversation. Especially, he gives me an idea that leads to Theorem \ref{ip2}. I also would like thank Prof.~Ruth Charney and Prof.~Thomas Koberda for their fruitful conversations on subgroups of right-angled Artin groups. I also thank Kevin Schreve for his help that improved the exposition of the paper. Lastly, I thank the referee for very helpful advice.

\section{Distortion of normal subgroups}

In this section, we review the concepts of group divergence and subgroup distortion. We investigate the connection between these concepts for a pair of finitely generated groups $(G,H)$, where $H$ is a normal subgroup of $G$. More precisely, we prove that the divergence of $G$ is a lower bound for the distortion of $H$ in $G$. We compute an upper bound for the distortion of $H$ in $G$ by dividing the group $G$ into a collection of smaller subgroups and computing the distortion of $H$ on these subgroups.

Before we define the concepts of group divergence and subgroup distortion, we need to define the notions of domination and equivalence. These notions are the key tool to measure group divergence and subgroup distortion.
\begin{defn}
Let $\mathcal{M}$ be the collection of all functions from $[0,\infty)$ to $[0,\infty]$. Let $f$ and $g$ be arbitrary elements of $\mathcal{M}$. \emph{The function $f$ is dominated by the function $g$}, denoted \emph{$f\preceq g$}, if there are positive constants $A$, $B$, $C$ and $D$ such that $f(x)\leq Ag(Bx)+Cx$ for all $x>D$. Two function $f$ and $g$ are \emph{equivalent}, denoted \emph{$f\sim g$}, if $f\preceq g$ and $g\preceq f$. 

\end{defn}

\begin{rem}
A function $f$ in $\mathcal{M}$ is \emph{linear, quadratic or exponential...} if $f$ is respectively equivalent to any polynomial with degree one, two or any function of the form $a^{bx+c}$, where $a>1, b>0$.
\end{rem}

\begin{defn}
Let $\{\delta_{\rho}\}$ and $\{\delta'_{\rho}\}$ be two families of functions of $\mathcal{M}$, indexed over $\rho \in (0,1]$. \emph{The family $\{\delta_{\rho}\}$ is dominated by the family $\{\delta'_{\rho}\}$}, denoted \emph{$\{\delta_{\rho}\}\preceq \{\delta'_{\rho}\}$}, if there exists constant $L\in (0,1]$ such that $\delta_{L\rho}\preceq \delta'_{\rho}$. Two families $\{\delta_{\rho}\}$ and $\{\delta'_{\rho}\}$ are \emph{equivalent}, denoted \emph{$\{\delta_{\rho}\}\sim \{\delta'_{\rho}\}$}, if $\{\delta_{\rho}\}\preceq \{\delta'_{\rho}\}$ and $\{\delta'_{\rho}\}\preceq \{\delta_{\rho}\}$.
\end{defn}

\begin{rem}
A family $\{\delta_{\rho}\}$ is dominated by (or dominates) a function $f$ in $\mathcal{M}$ if $\{\delta_{\rho}\}$ is dominated by (or dominates) the family $\{\delta'_{\rho}\}$ where $\delta'_{\rho}=f$ for all $\rho$. The equivalence between a family $\{\delta_{\rho}\}$ and a function $f$ in $\mathcal{M}$ can be defined similarly. Thus, a family $\{\delta_{\rho}\}$ is linear, quadratic, exponential, etc if $\{\delta_{\rho}\}$ is equivalent to the function $f$ where $f$ is linear, quadratic, exponential, etc.
\end{rem}

\begin{defn}
Let $G$ be a group with a finite generating set $S$ and $H$ a subgroup of $G$ with a finite generating set $T$. 
The \emph{subgroup distortion} of $H$ in $G$ is the function $Dist^H_G\!:(0,\infty)\to(0,\infty)$ defined as follows:
\[Dist^H_G(r)=\max \bigset{\abs{h}_T}{h\in H, \abs{h}_S\leq r}.\] 
\end{defn}

\begin{rem}
It is well-known that distortion does not depend on the choice of finite generating sets.
\end{rem}

We now recall Gersten's definition of divergence from \cite{MR1254309}. Let $X$ be a geodesic space and $x_0$ one point in $X$. Let $B_r(x_0)$ be the open ball centered at $x_0$ and radius $r$ and $S_r(x_0)$ the sphere with the same center and radius. Let $d_{r,x_0}$ be the induced length metric on the complement of the open ball with radius $r$ about $x_0$. If the point $x_0$ is clear from context, we will use the notation $d_r$ instead of using $d_{r,x_0}$.

\begin{defn}
Let $X$ be a geodesic space and $x_0$ one point in $X$. We define the \emph{divergence} of $X$, denoted $Div(X)$, as a family $\{\delta_\rho\}_{\rho \in (0,1]}$ of functions, where each function $\delta_\rho:[0, \infty)\to [0, \infty)$ is given as follows: 

For each $r$, let $\delta_\rho (r)=\sup d_{\rho r}(x_1,x_2)$ where the supremum is taken over all $x_1, x_2 \in S_r(x_0)$ such that $d_{\rho r}(x_1, x_2)<\infty$.
\end{defn}

\begin{rem}
The divergence of a geodesic space $X$ does not depend on the choice of $x_0$. Moreover, the concept of divergence is a quasi-isometry invariant. The \emph{divergence of a finitely generated group} $G$, denoted $Div(G)$, is the divergence of a Cayley graph $\Gamma(G,S)$ for some finite generating set $S$.
\end{rem}

Before we examine the connection between the group divergence and the distortion of a normal subgroup, we need to examine the relationship between the distance between cosets and the distance on the quotient group as follows:

\begin{prop}
\label{p1}
Let $G$ be a finitely generated group with a finite generating set $S=\{s_i\}_{i\in I}$ and $H$ a normal subgroup of $G$. For each element $g$ in $G$ let $\bar{g}$ be the image of $g$ under the canonical projection $p\!: G \to G/H$. Then the set $\bar{S}= \{\bar{s}_i\}_{i\in I}$ generates the quotient group $G/H$ and $d_S(aH,bH)=d_{\bar{S}}(\bar{a},\bar{b})$ for every $a$ and $b$ in $G$.
\end{prop}
\begin{proof}
It is obvious that $\bar{S}$ is a generating set of $G/H$ and we now prove that $d_S(aH,bH)=d_{\bar{S}}(\bar{a},\bar{b})$ every $a$ and $b$ in $G$. Let $ah_1$ and $bh_2$ be elements in $aH$ and $bH$ respectively such that the distance between $ah_1$ and $bh_2$ is the same as the distance between $aH$ and $bH$. Let $w=s_1s_2\cdots s_n$ be a word in $S$ such that $n=d_S(aH,bH)$ and $ah_1=(bh_2)(s_1s_2\cdots s_n)$. This implies that $\bar{a}=\bar{b}(\bar{s}_1\bar{s}_2\cdots \bar{s}_n)$. Therefore, $d_{\bar{S}}(\bar{a},\bar{b}) \leq n=d_S(aH,bH)$. Conversely, let $m=d_{\bar{S}}(\bar{a},\bar{b})$ and $w'=s'_1s'_2\cdots s'_m$ be a word in $S$ such that $\bar{a}=\bar{b}(\bar{s'}_1\bar{s'}_2\cdots \bar{s'}_m)$. Therefore, $c=b(s_1s_2\cdots s_m)$ is an element in $aH$ and $d_S(b,c)\leq m$. Therefore, $d_S(aH,bH)\leq m=d_{\bar{S}}(\bar{a},\bar{b})$. Thus, $d_S(aH,bH)=d_{\bar{S}}(\bar{a},\bar{b})$ for every $a$ and $b$ in $G$.
\end{proof}

\begin{thm}
\label{th1}
Let $G$ be a finitely generated group and $H$ a finitely generated infinite index infinite normal subgroup of $G$. Then the divergence $Div(G)$ of $G$ is dominated by the subgroup distortion $Dist_G^H$ of $H$ in $G$.
\end{thm}
\begin{proof}
Let $S=\{s_i\}_{i\in I}$ be a finite generating set of $G$. We assume that $S$ also contains a finite generating set $T$ of $H$. Therefore, we can consider the Cayley graph $\Gamma(H,T)$ as a subgraph of $\Gamma(G,S)$. For each element $g$ in $G$ let $\bar{g}$ be the image of $g$ under the canonical projection $p\!: G \to G/H$. By Proposition \ref{p1}, the set $\bar{S}= \{\bar{s}_i\}_{i\in I}$ generates the quotient group $G/H$ and $d_S(aH,bH)=d_{\bar{S}}(\bar{a},\bar{b})$ for every $a$ and $b$ in $G$.

Let $Div(G)=\{\delta_\rho\}_{\rho \in (0,1]}$ and we will prove that $\delta_\rho \preceq Dist^G_H$ for each $\rho$ in $(0,1/2]$. Let $r$ be an arbitrary positive number. We can assume that $r$ is an integer. Let $g$ be an element in $G$ such that $d_S(H,gH)=\abs{g}_S=r$. For each $x$ in $S_r(e)$ we are going to construct a path $\ell$ outside the open ball $B_{r/2}(e)$ connecting $x$ and $g$ with the length at most $2Dist^G_H(32r)+14r$.

Let $\alpha$ be an arbitrary bi-infinite geodesic containing $x$. Let $\alpha_1$ and $\alpha_2$ be two rays obtained from $\alpha$ with the initial point $x$. Since the distance between $x$ and $e$ is exactly $r$, either of $\alpha_1$ or $\alpha_2$ must lie outside the open ball $B_{r/2}(e)$. We assume that $\alpha_1$ lies outside the open ball $B_{r/2}(e)$. Let $x_1$ be the point on $\alpha_1$ such that the distance between $x$ and $x_1$ is exactly $4r$ and let $\ell_1$ the subsegment of $\alpha_1$ connecting $x$ and $x_1$. Therefore, $\ell_1$ must lie outside the open ball $B_{r/2}(e)$. 

We now find a point $x_2$ in $G$ such that $d_S(x_2H,H)\geq r$ and a path $\ell_2$ outside the open ball $B_{r/2}(e)$ connecting $x_1$ and $x_2$ with length at most $2r$. If $d_S(x_1H,H)\geq r$, then we let $x_2=x_1$ and $\ell_2$ the degenerate segment of length 0 at $x_2$. We now assume that $d_S(x_1H,H)<r$. Let $u$ be an element in $G$ such that $d_{\bar{S}} (\bar{x}_1,\bar{u})=2r$. Since the distance between $\bar{e}$ and $\bar{x}_1$ is at most $r$ with respect to the metric $d_{\bar{S}}$, the distance between $\bar{e}$ and $\bar{u}$ is at least $r$. Let $w=s_1s_2\cdots s_{2r}$ be a word in $S$ such that $\bar{u}=\bar{x_1w}$. Let $x_2=x_1w$ and $\ell_2$ the path connecting $x_1$ and $x_2$ with label by $w$. Then the length of $\ell_2$ is at most $2r$. Also, $d_S(e,x_1)\geq d_S(x,x_1)-d_S(e,x)\geq 4r-r\geq 3r$. Therefore, $\ell_2$ must lie outside the open ball $B_{r/2}(e)$. Obviously, $d_S(x_2H,H)=d_{\bar{S}} (\bar{x}_2, \bar {e})\geq r$. We also see that the distance between $x_2$ and $e$ is at most $7r$ and the distance between $x_2$ and $g$ is at most $8r$.

Let $h$ be an element in $H$ of length $16r$ with respect to the metric $d_S$ and let $\eta$ is a path in $\Gamma(H,T)$ connecting $e$ and $h$ with length at most $Dist^G_H(16r)$. Let $x_3=x_2h$ and $\ell_3 = x_2 \eta$. Then $\ell_3$ is a path connecting $x_2$ and $x_3$ with length at most $Dist^G_H(16r)$. Since all vertices of $\ell_3$ lie in $x_2H$, $\ell_3$ must lie outside the open ball $B_{r/2}(e)$. We observe that the distance between $x_2$ and $x_3$ is exactly $16r$ and the distance between $x_3$ and $e$ lies between $9r$ and $23r$.

Since $d_S(x_2H,gH)\leq d_S(x_2,g)\leq 8r$ and $x_3$ lies in $x_2H$, there is a point $x_4$ in $gH$ such that the distance between $x_3$ and $x_4$ is at most $8r$. Let $\ell_4$ be a geodesic connecting $x_3$ and $x_4$. Then the length of $\ell_4$ is at most $8r$. Since the distance between $x_3$ and $e$ is at least $9r$, the path $\ell_4$ must lie outside the open ball $B_{r/2}(e)$. We also see that the distance between $x_4$ and $e$ is at most $31r$ and the distance between $x_4$ and $g$ is at most $32r$. Therefore, there is a path $\ell_5$ with all vertices in $gH$ connecting $x_4$ and $g$ such that the length of $\ell_5$ is bounded above by $Dist^G_H(32r)$. Since the distance between $H$ and $gH$ is exactly $r$, $\ell_5$ must lie outside the open ball $B_{r/2}(e)$. 

Finally, let $\ell=\ell_1\cup \ell_2 \cup \ell_3 \cup \ell_4 \cup \ell_5$ then $\ell$ is a path outside $B_{r/2}(e)$ connecting $x$ and $g$ of length at most $Dist^G_H(32r)+Dist^G_H(16r)+14r$. Also, the distortion function $Dist^G_H$ is increasing. Then the length of $\ell$ is bounded above by $2Dist^G_H(32r)+14r$. Thus, any two points in the sphere $S_r(e)$ can be connected by a path outside $B_{r/2}(e)$ with length at most $4Dist^G_H(32r)+28r$. This implies that $\delta_\rho(r)$ is bounded above by $4Dist^G_H(32r)+28r$ for each $\rho$ in $(0,1/2]$. Therefore, the divergence $Div(G)$ of $G$ is dominated by the subgroup distortion $Dist_G^H$ of $H$ in $G$.

\end{proof}

The above theorem provides us with a lower bound for the distortion of a normal subgroup. The following theorem helps us compute an upper bound.

\begin{thm}
\label{theo1}
Let $G$ be a finitely generated group and $H$ normal subgroup of $G$ with the canonical projection $p:G \rightarrow G{/}H$. Suppose there is a finite collection $\mathcal{A}$ of finitely generated subgroups of $G$ that satisfies the following conditions:
\begin{enumerate}
\item The union of all subgroups in $\mathcal{A}$ generates $G$ and the subgroup $A\cap H$ is finitely generated for each subgroup $A$ in $\mathcal{A}$.
\item For each pair $A$ and $A'$ in $\mathcal{A}$, there is a sequence of subgroups in $\mathcal{A}$ $A=A_1, A_2, \cdots, A_n=A'$ such that $p(A_i\cap A_{i+1})$ is a finite index subgroup of $G/H$ for each $i$ in $\{1,2, \cdots, n-1\}$.
\end{enumerate}
Then, the subgroup $H$ is finitely generated and the distortion $Dist_{G}^H$ is dominated by the function $f(n)=n \max \set{Dist_{A}^{A\cap H}(n)}{A \in \mathcal{A}}$.
\end{thm}

\begin{proof}
We define function $g=\max \set{Dist_{A}^{A\cap H}}{A \in \mathcal{A}}$. Let $K$ be the number of elements of $\mathcal{A}$. For each subgroup $A$ in $\mathcal{A}$ we fix a finite generating set $S_A$ for $A$. Let $S=\bigcup_{A\in \mathcal{A}}S_A$ and $\bar{S}=p(S)$. Then $S$ is a finite generating set for $G$ and $\bar{S}$ is a finite generating set for $G/H$. Moreover, $\abs{p(g)}_{\bar{S}}\leq \abs{g}_S$ for each $g$ in $G$. Let $F$ be the intersection of all $p(A\cap A')$, where $A$ and $A'$ are elements in $\mathcal{A}$ such that $p(A\cap A')$ is a finite index subgroup of $G/H$. Therefore, $F$ is a finite index subgroup of $G/H$. Thus, there is a finite generating set $J$ of $F$ and a positive constant $M$ such that for each $v$ in $F$
\[(1/M)\abs{v}_{\bar{S}}-1\leq \abs{v}_J\leq M\abs{v}_{\bar{S}}+M.\]
Also, there is a finite set $I$ of group elements in $G$ that contains the identity element $e$ of $G$ such that 
\[G/H=\bigcup_{a\in I} p(a) F.\]
Let $C$ be the maximum on the lengths of all elements in $I$ with respect to the finite generating set $S$.

For each $b$, $b'$ in $I$ and $A$ in $\mathcal{A}$ such that $bAb'^{-1} \cap H\neq \emptyset$, we fix a group element $a_{(b,b',A)}$ in $A$ such that $ba_{(b,b',A)}b'^{-1}$ belongs to $H$. Let
\[L_1=\max \set{\abs{a_{(b,b',A)}}_{S_A}}{bAb'^{-1} \cap H\neq \emptyset, b,b'\in I, A \in \mathcal{A}}\]
and 

\[T_1=\set{a_{(b,b',A)}}{bAb'^{-1} \cap H\neq \emptyset, b,b'\in I, A \in \mathcal{A}}.\]

For each subgroup $A$ in $\mathcal{A}$ we fix a finite generating set $T_A$ for $A\cap H$. Let $T_2$ be the union of all sets $bT_Ab^{-1}$, where $b$ is an element in $I$ and $A$ is an element in $\mathcal{A}$. Let $T=T_1\cup T_2$ and we will prove that $T$ is a finite generating set of $H$. Moreover, we will use the pair of generating sets $(S,T)$ to compute the distortion of $H$ in $G$.

For each element $v$ in $J$ and $A$, $A'$ in $\mathcal{A}$ such that $p(A\cap A')$ is a finite index subgroup of $G/H$ we choose an element $u_{v,A,A'}$ in $A\cap A'$ such that $p(u_{v,A,A'})=v$. We let \[\ell(v, A, A')=\max\{\abs{u_{v,A,A'}}_{S_A}, \abs{u_{v,A,A'}}_{S_{A'}}\}\] and \[L_2=\max\set{\ell(v, A, A')}{v\in J\text{ and } A, A' \in \mathcal{A} \text{ such that } [G/H:p(A\cap A')]<\infty}.\]

It is not hard to see that for each element $q$ in $F$ and $A$, $A'$ in $\mathcal{A}$ such that $p(A\cap A')$ is a finite index subgroup of $G/H$ we can choose an element $g$ in $A\cap A'$ such that $p(g)=q$ and $\abs{g}_{S_A}$, $\abs{g}_{S_{A'}}$ are both bounded above by $L_2\abs{q}_J$. Therefore, both $\abs{g}_{S_A}$, $\abs{g}_{S_{A'}}$ are bounded above by $L_2\bigl(M\abs{q}_{\bar{S}}+M\bigr)$.

Let $n$ be an arbitrary integer and $h$ element in $H$ such that $\abs{h}_S \leq n$. Then there is a word $w=w_1w_2\cdots w_m$ in $S$ that represents element $h$ with the following properties:
\begin{enumerate}
\item The number $m$ is at most $Kn$ and each $w_i$ is a word in the finite generating set $S_{A_i}$ of some group $A_i$ in $\mathcal{A}$. (We accept that some of subword $w_i$ may be empty.)
\item For each $i$ in $\{1,2, 3, \cdots, m-1\}$, $p(A_i\cap A_{i+1})$ is a finite index subgroup of $G/H$.
\item The sum $\ell(w_1)+\ell(w_2)+\cdots+\ell(w_m)$ is bounded above by the number $n$.
\end{enumerate}

For each $i$ in $\{1,2,\cdots, m-1\}$ let $b_i$ in $I$ and such that $p(w_1 w_2 \cdots w_i)$ is an element in $p(b_i)F$. Let $u_i$ in $A_i \cap A_{i+1}$ such that $p(u_i)=p(b_i^{-1}w_1 w_2 \cdots w_i)$ and both $\abs{u_i}_{S_{A_i}}$, $\abs{u_i}_{S_{A_{i+1}}}$ are bounded above by $L_2\bigl(M \abs{p(b_i^{-1}w_1 w_2 \cdots w_i)}_{\bar{S}}+~M\bigr)$. Also, $\abs{p(b_i^{-1}w_1 w_2 \cdots w_i)}_{\bar{S}}\leq \abs{b_i^{-1}w_1 w_2 \cdots w_i}_S\leq n+C$. Then both $\abs{u_i}_{S_{A_i}}$, $\abs{u_i}_{S_{A_{i+1}}}$ are bounded above by $L_2M(n+C+1)$. It is obvious that
\[p(w_1 w_2 \cdots w_i)=p(b_iu_i).\]

We let $h_1=w_1u_1^{-1}b_1^{-1}$, $h_m=b_{m-1}u_{m-1}w_m$, and $h_i=b_{i-1}(u_{i-1}w_i u_i^{-1})b_i^{-1}$ for each $i$ in $\{2,3,\cdots, m-1\}$. Obviously, 
\[h=h_1h_2\cdots h_m \text{ and } p(h_i)=e_{G/H} \text{ for each $i$ in $\{1,2,\cdots, m-1\}$}.\]
Therefore, $h_i$ is an element of $H$ for each $i$ in $\{1,2,3,\cdots, m\}$. We now prove that the length of each element $h_i$ with respect to $T$ is bounded above by $g\bigl((2L_2M+1)n+ 2L_2M(C+1)+L_1\bigr)+1$.

Since $h_1$ is an element $A_1b_1^{-1}\cap H$, the set $A_1b_1^{-1}\cap H$ is not empty. Let $a_{(e,b_1,A_1)}$ be the element in $A_1$ defined as above such that $x_1=a_{(e,b_1,A_1)}b_1^{-1}$ is an element in $H$. Therefore, $h_1x_1^{-1}=w_1u_1^{-1}a_{(e,b_1,A_1)}^{-1}$ is an element in $A_1\cap H$ and \[\abs{w_1u_1^{-1}a_{(e,b_1,A_1)}^{-1}}_{S_{A_1}}\leq n+2L_2M(n+C+1)+L_1\leq (2L_2M+1)n+ 2L_2M(C+1)+L_1.\] 
Therefore, \[\abs{h_1x_1^{-1}}_T\leq \abs{w_1u_1^{-1}a_{(e,b_1,A_1)}^{-1}}_{T_{A_1}}\leq Dist_{A_1}^{A_1\cap H}\bigl((2L_2M+1)n+ 2L_2M(C+1)+L_1\bigr).\]
This implies that \[\abs{h_1x_1^{-1}}_T\leq g\bigl((2L_2M+1)n+ 2L_2M(C+1)+L_1\bigr).\]
Also, $x_1$ is an element in $T$. Then \[\abs{h_1}_T\leq g\bigl((2L_2M+1)n+ 2L_2M(C+1)+L_1\bigr)+1.\]

Since $h_i$ is an element $b_{i-1}A_ib_i^{-1}\cap H$ for each $i$ in $\{2,3,\cdots, m-1\}$, the set $b_{i-1}A_ib_i^{-1}\cap H$ is not empty. Let $a_{(b_{i-1},b_i,A_i)}$ be the element in $A_i$ defined as above such that $x_i=b_{i-1}a_{(b_{i-1},b_i,A_i)}b_i^{-1}$ is an element in $H$. Therefore, $h_ix_i^{-1}=b_{i-1}(u_{i-1}w_iu_i^{-1}a_{(b_{i-1},b_i,A_i)}^{-1})b_{i-1}^{-1}$ is an element in $b_{i-1}A_ib_{i-1}^{-1}\cap H$ and \[\abs{u_{i-1}w_iu_i^{-1}a_{(b_{i-1},b_i,A_i)}^{-1}}_{S_{A_i}}\leq n+2L_2M(n+C+1)+L_1\leq (2L_2M+1)n+ 2L_2M(C+1)+L_1.\] 
Therefore, \[\abs{h_ix_i^{-1}}_T\leq \abs{u_{i-1}w_iu_i^{-1}a_{(b_{i-1},b_i,A_i)}^{-1}}_{T_{A_i}}\leq Dist_{A_i}^{A_i\cap H}\bigl((2L_2M+1)n+ 2L_2M(C+1)+L_1\bigr).\]
This implies that \[\abs{h_ix_i^{-1}}_T\leq g\bigl((2L_2M+1)n+ 2L_2M(C+1)+L_1\bigr).\]
Also, $x_i$ is an element in $T$. Then \[\abs{h_i}_T\leq g\bigl((2L_2M+1)n+ 2L_2M(C+1)+L_1\bigr)+1.\]

Since $h_m$ is an element $b_{m-1}A_m\cap H$, the set $b_{m-1}A_m\cap H$ is not empty. Let $a_{(b_{m-1},e,A_m)}$ be the element in $A_m$ defined as above such that $x_m=b_{m-1}a_{(b_{m-1},e,A_m)}$ is an element in $H$. Therefore, $h_mx_m^{-1}=b_{m-1}(u_{m-1}w_ma_{(b_{m-1},e,A_m)}^{-1})b_{m-1}^{-1}$ is an element in $b_{m-1}A_mb_{m-1}^{-1}\cap H$ and \[\abs{u_{m-1}w_ma_{(b_{m-1},e,A_m)}^{-1}}_{S_{A_m}}\leq n+2L_2M(n+C+1)+L_1\leq (2L_2M+1)n+ 2L_2M(C+1)+L_1.\] 
Therefore, \[\abs{h_mx_m^{-1}}_T\leq \abs{u_{m-1}w_ma_{(b_{m-1},e,A_m)}^{-1}}_{T_{A_m}}\leq Dist_{A_m}^{A_m\cap H}\bigl((2L_2M+1)n+ 2L_2M(C+1)+L_1\bigr).\]
This implies that \[\abs{h_mx_m^{-1}}_T\leq g\bigl((2L_2M+1)n+ 2L_2M(C+1)+L_1\bigr).\]
Also, $x_m$ is an element in $T$. Then \[\abs{h_m}_T\leq g\bigl((2L_2M+1)n+ 2L_2M(C+1)+L_1\bigr)+1.\]

Therefore, the group element $h$ can be represented by a word in $T$ and $\abs{h}_T$ is bounded above by $Kn g\bigl((2L_2M+1)n+ 2L_2M(C+1)+L_1\bigr)+Kn$. This implies that the subgroup $H$ is finitely generated and the distortion $Dist_{G}^H$ is dominated by the function $f(n)$.
\end{proof}

\begin{cor}
Let $G$ be the fundamental group of a simply connected finite graph of finitely generated groups and $H$ a finitely generated normal subgroup of $G$ not contained in any edge subgroup. We assume that the intersection of $H$ with any vertex subgroup is finitely generated. Then, the distortion $Dist_{G}^H$ is dominated by the function $f(n)=n \max \set{Dist_{A}^{A\cap H}(n)}{A \text{ is a vertex subgroup}}$.
\end{cor}

\begin{proof}
Let $p:G \rightarrow G{/}H$ be the canonical projection. Let $\mathcal{A}$ be the collection of all vertex subgroups of $G$. Since $G$ is the fundamental group of a simply connected finite graph of groups, the group $G$ is generated by the union of all vertex subgroups in $\mathcal{A}$. Also, the intersection of $H$ with any vertex subgroup is finitely generated. Therefore, $\mathcal{A}$ and $H$ satisfy condition (1) of Theorem \ref{theo1}. By Theorem 10 in \cite{MR0260879}, the image of each edge subgroup of $G$ under $p$ has finite index in $G{/}H$. Thus, $\mathcal{A}$ and $H$ satisfy condition (2) of Theorem \ref{theo1}. 
Therefore, the distortion $Dist_{G}^H$ is dominated by the function $f(n)=n \max \set{Dist_{A}^{A\cap H}(n)}{A \text{ is a vertex subgroup of $G$}}$. 
\end{proof}

One application of Theorem \ref{th1} and Theorem \ref{theo1} is to compute the distortion of collection of normal subgroups of right-angled Artin groups which all induce infinite cyclic quotient group. We first recall the concept of right-angled Artin groups and then state the result on divergence of right-angled Artin groups proved by Behrstock-Charney \cite{MR2874959}. We will use the divergence of right-angled Artin groups for the lower bound of the distortion we are working.

\begin{defn}[Right-angled Artin groups]
For each $\Gamma$ a finite simplicial graph the associated \emph{right-angled Artin group} $A_{\Gamma}$ has generating set $S$ the vertices of $\Gamma$, and relations $st = ts$ whenever $s$ and $t$ are adjacent vertices.
\end{defn}

\begin{thm}[Behrstock-Charney \cite{MR2874959}]
\label{thbc}
Let $\Gamma$ be a connected graph with at least 2 vertices. $A_\Gamma$ has linear divergence if and only if $\Gamma$ is a join; otherwise its divergence is quadratic.
\end{thm}

\begin{cor}
\label{c1}
Let $\Gamma$ be a simplicial connected graph with at least two vertices. Let $p$ be an arbitrary group homomorphism from $A_\Gamma$ to $\Z$ and $N$ the kernel of the map $p$. Let $\Gamma'$ be the induced subgraph of $\Gamma$ with the vertices which are mapped to nonzero numbers in $\Z$ by $p$. If the graph $\Gamma'$ is connected and the star of $\Gamma'$ in $\Gamma$ contains all vertices of $\Gamma$, then the subgroup $N$ is finitely generated and the distortion of $N$ in $A_\Gamma$ is at most quadratic. Moreover, if each vertex of $\Gamma$ is mapped to a nonzero number in $\Z$, then the distortion of the $N$ in $A_\Gamma$ is linear when $\Gamma$ is a join and the distortion is quadratic otherwise. 
\end{cor}

\begin{proof}
We can assume that the map $p$ is surjective. We let $\mathcal{A}$ is the collection of all subgroups $A_e$, where each $A_e$ is the right-angled Artin subgroup induced by some edge $e$ with at least one endpoint in $\Gamma'$. We can verify easily that the collection $\mathcal{A}$ and the subgroup $N$ satisfy all hypothesis of Theorem \ref{theo1}. Also each subgroup $A_e$ is isomorphic to $\Z^2$. Therefore, it is not hard to see that the subgroup $N \cap A_e$ is finitely generated and the distortion of $N \cap A_e$ in $A_e$ is linear. Therefore, the subgroup $N$ is finitely generated and the distortion of the $N$ in $A_\Gamma$ is at most quadratic. 

We now assume that each vertex of $\Gamma$ is mapped to a nonzero number in $\Gamma$. If $\Gamma$ is not a join, then the divergence of $\Gamma$ is quadratic (see Theorem \ref{thbc}). Also, the subgroup $N$ is a finitely generated normal subgroup of $A_\Gamma$. Therefore, the distortion of $H_\Gamma$ in $A_\Gamma$ is exactly quadratic. 

We now assume that $\Gamma$ is a join of $\Gamma_1$ and $\Gamma_2$. We will prove that the distortion of the $N$ in $A_\Gamma$ is linear. It is a well-known result that $A_\Gamma$ is the direct product of $A_{\Gamma_1}$ and $A_{\Gamma_2}$. Since $p(A_{\Gamma_1})$ and $p(A_{\Gamma_2})$ are nontrivial subgroups of $\Z$, there are finite index subgroup $A_1$ of $A_{\Gamma_1}$ and finite index subgroup $A_2$ of $A_{\Gamma_2}$ such that $p(A_1)=p(A_2)=p(A_1\times A_2)$. Let $p_1=~p_{|A_1 \times A_2}:~A_1 \times A_2~\rightarrow p(A_1\times A_2)$ and $N_1$ is the kernel of $p_1$. It is obvious that $A_1\times A_2$ is a finite index subgroup of $A_\Gamma$ and $N_1$ is a finite index subgroup of $N$. Therefore, it is sufficient to prove that the distortion of $N_1$ in $A_1 \times A_2$ is linear.

In fact, let $S_1$ be a finite generating set of $A_1$ and $S_2$ a finite generating set of $A_2$. Then, $S=S_1\cup S_2$ is a finite generating set of $A_1 \times A_2$. We can assume that $p(A_1\times A_2)=\Z$ and there are elements $a\in S_1$, $b\in S_2$ such that $p_1(a)=p_1(b)=1$. Let $T_1$ be the set of all elements of the form $sb^{-p_1(s)}$ where each $s$ is element in $S_1$. Similarly, let $T_2$ be the set of all elements of the form $sa^{-p_1(s)}$ where each $s$ is element in $S_2$. We will prove that $T=T_1\cup T_2$ is a finite generating set of $N_1$. Moreover, we will use finite generating sets $(S,T)$ to prove the distortion of $N_1$ in $A_1 \times A_2$ is linear.

Let $K=\max_{s\in S}{\abs{p_1(s)}}$. Let $n$ be an arbitrary positive integer and $h$ be an arbitrary element in $N_1$ such that $\abs{h}_S\leq n$. Then we can write $h=(a^{m_1}_1a^{m_2}_2\cdots a^{m_k}_k)(b^{n_1}_1b^{n_2}_2\cdots b^{n_\ell}_\ell)$ such that:

\begin{enumerate}
\item Each $a_i$ is element in $S_1$ and each $b_j$ is element in $S_2$.
\item $\bigl(\abs{m_1}+\abs{m_2}+\cdots+\abs{m_k}\bigr)+\bigl(\abs{n_1}+\abs{n_2}+\cdots+\abs{n_\ell}\bigr)\leq n$.
\item $\bigl(m_1p_1(a_1)+m_2p_1(a_2)+\cdots+m_kp_1(a_k)\bigr)+\bigl(n_1p_1(b_1)+n_2p_1(b_2)+\cdots+n_\ell p_1(b_\ell)\bigr)=0$
\end{enumerate}
Let $m=m_1p_1(a_1)+m_2p_1(a_2)+\cdots+m_kp_1(a_k)$. Then, $n_1p_1(b_1)+n_2p_1(b_2)+\cdots+n_\ell p_1(b_\ell)=-m$ and $\abs{m}\leq Kn$. Since $a$ commutes with each $b_j$, $b$ commutes with each $a_i$ and $a$, $b$ commute, we can rewrite $h$ as follows: 
\[h=(a^{m_1}_1a^{m_2}_2\cdots a^{m_k}_kb^{-m})(ba^{-1})^m(a^mb^{n_1}_1b^{n_2}_2\cdots b^{n_\ell}_\ell).\]
Also, 
\[a^{m_1}_1a^{m_2}_2\cdots a^{m_k}_kb^{-m}=(a_1b^{-p_1(a_1)})^{m_1}(a_2b^{-p_1(a_2)})^{m_2}\cdots (a_kb^{-p_1(a_k)})^{m_k}\]
\[a^mb^{n_1}_1b^{n_2}_2\cdots b^{n_\ell}_\ell=(b_1a^{-p_1(b_1)})^{n_1}(b_2a^{-p_1(b_2)})^{n_2}\cdots(b_\ell a^{-p_1(b_\ell)})^{n_\ell}\]

and $a_ib^{-p_1(a_i)}$, $b_ja^{-p_1(b_j)}$ and $ba^{-1}$ all belong to $T$. Therefore, 
\begin{align*}
\abs{h}_T\leq \bigl(\abs{m_1}+\abs{m_2}+\cdots+\abs{m_k}\bigr)+\bigl(\abs{n_1}+\abs{n_2}+\cdots+\abs{n_\ell}\bigr)+\abs{m}\leq (K+1)n.
\end{align*}
Therefore, the distortion function $Dist^{N_1}_{A_1\times A_2}$ is bounded above by $(K+1)n$.

\end{proof}

\begin{rem}
If the map $p$ in the above corollary maps each vertex of $\Gamma$ to 1, then the subgroup $N$ is the Bestvina-Brady subgroup of $A_\Gamma$ (see \cite{MR1465330}). We remark that the distortion of Bestvina-Brady subgroups were also computed by the author in \cite{Tran}.

In Corollary \ref{c1}, the hypothesis that the graph $\Gamma'$ is connected and its star in $\Gamma$ contains all vertices of $\Gamma$ is also a necessary condition for the subgroup $N$ to be finitely generated (see Theorem 6.1 in \cite{MR1337468}).

In the above corollary, we observe that if each vertex of $\Gamma$ is mapped to a nonzero number in $\Z$, then the distortion of the subgroup $N$ and the divergence of the whole group $A_\Gamma$ are equivalent (they are both linear or both quadratic). However, this fact is no longer true in general when the image of some vertex of $\Gamma$ under $p$ is zero (see Proposition~\ref{pp1}).

The normal subgroups of right-angled Artin groups in Corollary \ref{c1} all induce infinite cyclic quotient groups. Therefore, we can use Theorem 4.1 in \cite{MR1254309} instead of Theorem \ref{th1} in this paper to prove the lower bound of the distortion of $N$. However, we will examine distortion of normal subgroups of right-angled Artin groups that induce non cyclic quotient groups in Section~\ref{GBBS}. In this case, we can only use Theorem \ref{th1} to compute the lower bound of the distortion.
\end{rem}

\begin{figure}
\begin{tikzpicture}[scale=2]


\draw (-5,1) node[circle,fill,inner sep=2pt](a){} -- (-6,0); \draw (-5,1) -- (-5,0); \draw (-5,1) -- (-4,0); \draw (-5,1) -- (-3,0);

\draw (-6,0) node[circle,fill,inner sep=2pt](a){} -- (-5,0) node[circle,fill,inner sep=2pt](a){} -- (-4,0) node[circle,fill,inner sep=2pt](a){} -- (-3,0) node[circle,fill,inner sep=2pt](a){}; 

\node at (-5,1.2) {$b$};\node at (-6,-0.2) {$a_1$};\node at (-5,-0.2) {$a_2$};\node at (-4,-0.2) {$a_3$};\node at (-3,-0.2) {$a_4$};
\end{tikzpicture}
\caption{The right angled-Artin group with the above defining graph has linear divergence but it contains a finitely generated normal subgroup with quadratic distortion}
\label{a1}
\end{figure}

\begin{prop}
\label{pp1}
Let $\Gamma$ be the graph in Figure \ref{a1}. Let $p$ be a group homomorphism from $A_\Gamma$ to $\Z$ that maps each $a_i$ to $1$ and $b$ to zero. Let $N$ be the kernel of the map $p$. Then the distortion of $N$ in $A_\Gamma$ is quadratic. 
\end{prop}
\begin{proof}
Let $\Gamma_1$ be the subgraph of $\Gamma$ induced by $a_1$, $a_2$, $a_3$, and $a_4$. Therefore, the group $A_\Gamma$ is the direct product of the subgroup $A_{\Gamma_1}$ and the cyclic subgroup generated by $b$. We also observe that the group $N$ is the direct product of the Bestvina-Brady subgroup $N_1$ of $A_{\Gamma_1}$ and the cyclic subgroup generated by $b$. Since the distortion of $N_1$ in $A_{\Gamma_1}$ is quadratic (see Corollary \ref{c1} or \cite{Tran}), the distortion of $N$ in $A_\Gamma$ is also quadratic.
\end{proof}

\begin{rem}
In the above proposition, the divergence of the whole group $A_\Gamma$ is linear but the distortion of $N$ in $A_\Gamma$ is quadratic. Therefore, Proposition \ref{pp1} provides an example of the non equivalent relation between the group divergence and normal subgroup distortion.
\end{rem}




\section{Macura's examples}
Macura \cite{MR3032700} introduced a class of $\CAT(0)$ groups (i.e. \emph{groups which act properly and cocompactly on some $\CAT(0)$ spaces}) to study group divergence. We revisit these groups and investigate the distortion of their normal subgroups to study a connection between group divergence and subgroup distortion. First, we will review the class of groups introduced by Macura \cite{MR3032700}. 

For each integer $d\geq2$, we define \[G_d=\langle a_0, a_1,\cdots, a_d |a_0a_1=a_1a_0, a^{-1}_ia_0a_i=a_{i-1}, \text{for $2\leq i\leq d$}\rangle.\] 
The divergence of each group $G_d$ is given by the following theorem. 

\begin{thm}\cite{MR3032700}
\label{th2}
For each integer $d\geq 2$ the divergence of each $G_d$ is a polynomial of degree $d$.
\end{thm}

We now construct a normal subgroup in each $G_d$ with the distortion the same as the divergence of $G_d$. We note that there is another presentation for each $G_d$ which is given by \[G_d=\langle x_1, x_2,\cdots, x_d, t |tx_1t^{-1}=x_1, tx_it^{-1}=x_ix^{-1}_{i-1}, \text{for $2\leq i\leq d$}\rangle.\]

We can see that the above two presentations of $G_d$ are equivalent by the map on generators $a_0\mapsto t^{-1}$, $a_i\mapsto t^{-1}x_i$ and its inverse $t\mapsto a^{-1}_0$, $x_i \mapsto a^{-1}_0a_i$. From the second presentation, we observe that the subgroup $F_d$ of $G_d$ generated by $x_1, x_2, \cdots, x_d$ is a free group of rank $d$. Moreover, the group endomorphism $\Phi_d$ that maps $x_1$ to $x_1$, $x_i$ to $x_ix^{-1}_{i-1}$ for each $2\leq i\leq d$ is a group automorphism with the inverse $\Psi_d$ that maps $x_i$ to $x_ix_{i-1}\cdots x_1$ for each $1\leq i\leq d$. Obviously, each $G_d$ is a free by cyclic group $F_d \rtimes_{\Phi_d} \Z$. 

The following proposition shows an upper bound on the distortion of each subgroups $F_d$ in $G_d$.

\begin{prop}
\label{p2}
For each integer $d\geq 2$ let $G_d$ and $F_d$ be groups defined as above. Then the subgroup distortion of each $F_d$ in $G_d$ is bounded above by a polynomial of degree $d$.
\end{prop}

The above proposition can be proved by the following three lemmas.

\begin{lem}[Gersten \cite{MR1254309}]
Let $G=F\rtimes_{\Phi} \Z$ be a free by cyclic group, where $F$ is a free group generated by a set $S$ of $d$ elements $x_1, x_2, \cdots, x_d$. For each positive number $n$ we define \[\phi(n)=\max\set{\abs{\Phi^{j}(x_i)}_S}{1\leq i\leq d, \abs{j}\leq n}.\]
Then the distortion of $F$ in $G$ is dominated by the function $n\phi(n)$.
\end{lem}

We now use the above proposition and the following two lemmas to show that the subgroup distortion of each $F_d$ in $G_d$ is bounded above by a polynomial of degree $d$.

\begin{lem}
Let $F_d$ be a free group generated by a set $S$ of $d$ elements $x_1, x_2, \cdots, x_d$. Let $\Phi_d$ be a free group automorphism of $F_d$ that maps $x_1$ to $x_1$ and $x_i$ to $x_ix^{-1}_{i-1}$ for each $2\leq i\leq d$. Then $\abs{\Phi^n_d(x_i)}_{S}$ is bounded above by $2n^{i-1}$ for each $1\leq i \leq d$ and each positive integer $n$.
\end{lem}
\begin{proof}
We will prove the above lemma by induction on $i$. For $i=1$ we have $\Phi^n_d(x_1)=x_1$. Therefore, $\abs{\Phi^n_d(x_1)}_{S}$ is exactly 1 and it is obviously bounded above by $2n^{i-1}$ for $i=1$. Assume that $\abs{\Phi^n_d(x_{i-1})}_{S}$ is bounded above by $2n^{i-2}$ for each positive integer $n$. We need to prove that $\abs{\Phi^n_d(x_i)}_{S}$ is bounded above by $2n^{i-1}$ for each positive integer $n$. We now prove this by induction on $n$.

For $n=1$ we have $\Phi_d(x_i)=x_ix^{-1}_{i-1}$. Therefore, $\abs{\Phi_d(x_i)}_{S}$ is exactly 2 and it is obviously bounded above by $2n^{i-1}$ for $n=1$. Assume that $\abs{\Phi^{n-1}_d(x_i)}_{S}$ is bounded above by $2(n-1)^{i-1}$. We have $\Phi^n_d(x_i)=\Phi^{n-1}_d(x_ix^{-1}_{i-1})= \Phi^{n-1}_d(x_i)\big[\Phi^{n-1}_d(x_{i-1})\big]^{-1}$. Therefore, $\abs{\Phi^n_d(x_i)}_{S}$ is bounded above by $2(n-1)^{i-2}+2(n-1)^{i-1}$. Also $2(n-1)^{i-2}+2(n-1)^{i-1}=2(n-1)^{i-2}n\leq 2n^{i-1}$. Thus, $\abs{\Phi^n_d(x_i)}_{S}$ is bounded above by $2n^{i-1}$.
\end{proof}

\begin{lem}
Let $F_d$ be a free group generated by a set $S$ of $d$ elements $x_1, x_2, \cdots, x_d$. Let $\Psi_d$ be a free group automorphism of $F_d$ that maps $x_i$ to $x_ix_{i-1}\cdots x_1$ for each $1\leq i\leq d$. Then $\abs{\Psi^n_d(x_i)}_{S}$ is bounded above by $in^{i-1}$ for each $1\leq i \leq d$ and each positive integer $n$.
\end{lem}

\begin{proof}
We will prove the above lemma by induction on $i$. For $i=1$ we have $\Psi^n_d(x_1)=x_1$. Therefore, $\abs{\Psi^n_d(x_1)}_{S}$ is exactly 1 and it is obviously bounded above by $in^{i-1}$ for $i=1$. Assume that each $\abs{\Psi^{n}_d(x_k)}_{S}$ is bounded above by $kn^{k-1}$ for each $k\leq i-1$ and each positive integer $n$. We need to prove that $\abs{\Psi^n_d(x_i)}_{S}$ is bounded above by $in^{i-1}$ for each positive integer $n$. We now prove this by induction on $n$.

For $n=1$ we have $\Psi_d(x_i)=x_ix_{i-1}\cdots x_1$. Therefore, $\abs{\Psi_d(x_i)}_{S}$ is exactly $i$ and it is obviously bounded above by $i n^{i-1}$ for $n=1$. Assume that $\abs{\Psi^{n-1}_d(x_i)}_{S}$ is bounded above by $i(n-1)^{i-1}$. We have $\Psi^n_d(x_i)=\Psi^{n-1}_d(x_ix_{i-1}\cdots x_1)= \Psi^{n-1}_d(x_i)\Psi^{n-1}_d(x_{i-1})\cdots \Psi^{n-1}_d(x_1)$. Also, each $\abs{\Psi^{n-1}_d(x_k)}_{S}$ is bounded above by $k(n-1)^{k-1}$ for each $k\leq i$. Therefore, $\abs{\Psi^n_d(x_i)}_{S}$ is bounded above by $1(n-1)^0+2(n-1)^1+3(n-1)^2+\cdots +i(n-1)^{i-1}$. Also 

\[1(n-1)^0+2(n-1)^1+3(n-1)^2+\cdots +i(n-1)^{i-1}\leq i n^{i-1}.\] 

Thus, $\abs{\Psi^n_d(x_i)}_{S}$ is bounded above by $in^{i-1}$.
\end{proof}

\begin{rem}
\label{remark1}
 Theorem \ref{th1} and Proposition \ref{p2} give an alternative proof of the upper bound for divergence of $G_d$ in Theorem \ref{th2}.
\end{rem}

\begin{thm}
\label{th3}
For each integer $d\geq 2$ let $G_d$ and $F_d$ be groups defined as above. Then the subgroup distortion of each $F_d$ in $G_d$ is a polynomial of degree $d$.
\end{thm}

The proof of the above theorem is given by Theorem \ref{th1} (or Theorem 4.1 in \cite{MR1254309}), Theorem \ref{th2}, and Proposition \ref{p2}. 

\section{On distortion of normal subgroups in right-angled Artin groups}
\label{GBBS}
In this section, we study the distortion of normal subgroups in right-angled Artin groups. Most of our computations are based on Theorem \ref{th1} and Theorem \ref{theo1}.

\begin{defn}
Let $\Gamma_1$ and $\Gamma_2$ be two graphs, the \emph{join} of $\Gamma_1$ and $\Gamma_2$ is a graph obtained by connecting every vertex of $\Gamma_1$ to every vertex of $\Gamma_2$ by an edge. A graph $\Gamma$ is \emph{join} if it can be decomposed as the join of its two subgraphs.
\end{defn}

\begin{defn}
Let $\Gamma$ be a graph with the vertex set $S$. 
\begin{enumerate}
\item A subgraph $\Gamma_1$ of $\Gamma$ is \emph{dominating} if for every vertex $v$ in $\Gamma-\Gamma_1$, $d(v, \Gamma_1)=1$. 
\item A subgraph $\Gamma_1$ is \emph{strongly dominating} if there is a vertex $u$ in $\Gamma_1$ that is adjacent to all vertices in $\Gamma-\Gamma_1$.
\item A subgraph $\Gamma_1$ is \emph{specially dominating} if there is a vertex $u$ in $\Gamma_1$ that is adjacent to all vertices in $\Gamma-\{u\}$. \item Let $S_1$ be a subset of the vertex set $S$ of $\Gamma$. The \emph{induced subgraph $\Gamma_1$} with vertex set $S_1$ is the subgraph with vertex set $S_1$ and all edges in $\Gamma$ with both endpoints in $S_1$.
\item Let $\mathcal{A}$ be a collection of subgraphs of $\Gamma$. A subgraph $\Gamma_1$ is \emph{generated by} $\mathcal{A}$ if it is the induced subgraph with the vertex set $S_1$ which is the union of vertex sets of all subgraphs in $\mathcal{A}$.
\item Two subgraphs $\Gamma_1$ and $\Gamma_2$ commutes if each vertex in $\Gamma_1$ is connected to all vertices in $\Gamma_2$ and vice versa.
\end{enumerate}
\end{defn}
\begin{figure}
\begin{tikzpicture}
\draw (-5,1) node[circle,fill,inner sep=2pt, color=red](a){} -- (-4,1) node[circle,fill,inner sep=2pt, color=red](a){} -- (-3,1) node[circle,fill,inner sep=2pt, color=red](a){};

\draw (-5, 1) -- (-5, 0); \draw (-5,1) -- (-4,0); \draw (-5,1) -- (-3,0); \draw (-4,1)-- (-3,0); \draw (-3,1) -- (-3,0);

\draw (-5,0) node[circle,fill,inner sep=2pt, color=blue](a){} -- (-4,0) node[circle,fill,inner sep=2pt, color=blue](a){} -- (-3,0) node[circle,fill,inner sep=2pt, color=blue](a){};

\draw (-1,1) node[circle,fill,inner sep=2pt, color=red](a){} -- (0,1) node[circle,fill,inner sep=2pt, color=red](a){} -- (1,1) node[circle,fill,inner sep=2pt, color=red](a){}; 

\draw (-1,1) -- (-1,0); \draw (-1,1) -- (0,0); \draw (0,1) -- (0,0); \draw (0,1) -- (1,0); \draw (1,1) -- (1,0);

\draw (-1,0) node[circle,fill,inner sep=2pt, color=blue](a){} -- (0,0) node[circle,fill,inner sep=2pt, color=blue](a){} -- (1,0) node[circle,fill,inner sep=2pt, color=blue](a){};

\draw (3,1) node[circle,fill,inner sep=2pt, color=red](a){} -- (4,1) node[circle,fill,inner sep=2pt, color=red](a){} -- (5,1) node[circle,fill,inner sep=2pt, color=red](a){}; 

\draw (3,1) -- (3,0); \draw (5,1) -- (5,0);

\draw (3,0) node[circle,fill,inner sep=2pt, color=blue](a){} -- (4,0) node[circle,fill,inner sep=2pt, color=blue](a){} -- (5,0) node[circle,fill,inner sep=2pt, color=blue](a){}; 

\end{tikzpicture}
\caption{Illustration of some right-angled Artin groups with non-join defining graphs and their generalised Bestvina-Brady subgroups.}
\label{fig1}
\end{figure}

\begin{defn}
Let $\Gamma$ be a simplicial graph and $A_{\Gamma}$ the associated right-angled Artin group. Let $p$ be a group epimorphism from $A_{\Gamma}$ to $\Z$. The \emph{living subgraph} $\mathcal{L}(p)$ of $\Gamma$ with respect to $p$ is the induced subgraph of $\Gamma$ with the vertices corresponding to generators which are not mapped to 0 by $p$.
\end{defn}

The following definition is a generalization of Bestvina-Brady subgroups in \cite{MR1465330}. 
\begin{defn}[Generalized Bestvina-Brady subgroups]
\label{d1}
Let $\Gamma$ be a finite simplicial graph with the vertex set $S$. Let $d$ be a positive integer that is less than or equal the number of elements of $S$. Label each vertex of $\Gamma$ with an integer in $\{1,2,\dots d\}$. We construct an \emph{epimorphism} $\Phi$ from $A_{\Gamma}$ to $\Z^d$ by sending each vertex with label $i$ to $e_i=(0,0,\cdots, 1^{i th}, \cdots, 0)$. The subgroup $H_{\Phi}=Ker (\Phi)$ is a \emph{generalized Bestvina-Brady subgroup} of $A_{\Gamma}$ with respect to $\Phi$.
\end{defn}
 
\begin{rem}
The subgroup $H_{\Phi}$ is the commutator subgroup if $d=\abs{S}$ and $H_{\Phi}$ is the Bestvina-Brady subgroup (see \cite{MR1465330}) if $d=1$.
\end{rem}

For each $1\leq i \leq d$ let $S_i$ be the set of vertices of $\Gamma$ mapped to $e_i$ by $\Phi$. Let $\Gamma_i$ be the induced subgraph of $\Gamma$ with vertex set $S_i$ (i.e.~$\Gamma_i$ is the union of all edges of $\Gamma$ with both endpoints in $S_i$). We call $\Gamma_i$ \emph{basis subgraphs} of $\Gamma$ with respect to $\Phi$.

\begin{thm}
Let $\Gamma$ be a finite simplicial graph. Let $H_\Phi$ be a generalized Bestvina-Brady subgroup of $A_{\Gamma}$. Then $H_\Phi$ is finitely generated iff each basis subgraph of $\Gamma$ with respect to $\Phi$ is connected and dominating. 
\end{thm}
\begin{proof}
Assume that the group $H_\Phi$ is finitely generated. For each $i$ in $\{1,2,\cdots d\}$ let $p_i$ be a group homomorphism from $\Z^d$ to $\Z$ that maps $e_i$ to 1 and each $e_j$ to 0 for $j\neq i$. It is obvious that the map $p_i\circ \Phi$ is a a group epimorphism from $A_{\Gamma}$ to $\Z$ and each basis subgraph $\Gamma_i$ is the living graph $\mathcal{L}(p_i\circ \Phi)$ of the group epimorphism $p_i\circ \Phi$. Therefore, each subgraph is connected and dominating by Proposition 6.2 in \cite{MR1337468}. The proof of the opposite direction can be deduced from the proof of Corollary 6.6 in the same article \cite{MR1337468}.
\end{proof}

\begin{figure}
\begin{tikzpicture}

\draw (-5,1) node[circle,fill,inner sep=2pt, color=red](a){} -- (-4,1) node[circle,fill,inner sep=2pt, color=red](a){};

\draw (-5,1) -- (-6,0); \draw (-5,1) -- (-5,0); \draw (-5,1) -- (-4,0); \draw (-5,1) -- (-3,0); \draw (-4,1) -- (-3,0);

\draw (-6,0) node[circle,fill,inner sep=2pt, color=blue](a){} -- (-5,0) node[circle,fill,inner sep=2pt, color=blue](a){} -- (-4,0) node[circle,fill,inner sep=2pt, color=blue](a){} -- (-3,0) node[circle,fill,inner sep=2pt, color=blue](a){};

\draw (-0.5,1) node[circle,fill,inner sep=2pt, color=red](a){} -- (0.5,1) node[circle,fill,inner sep=2pt, color=red](a){};

\draw (-0.5,1) -- (-1,0); \draw (0.5,1) -- (-1,0); \draw (0.5,1) -- (0,0); \draw (0.5,1) -- (1,0); 

\draw (-1,0) node[circle,fill,inner sep=2pt, color=blue](a){} -- (0,0) node[circle,fill,inner sep=2pt, color=blue](a){} -- (1,0) node[circle,fill,inner sep=2pt, color=blue](a){};

\draw (3.5,1) node[circle,fill,inner sep=2pt, color=red](a){} -- (4.5,1) node[circle,fill,inner sep=2pt, color=red](a){};

\draw (3.5,1) -- (3,0); \draw (3.5,1) -- (4,0); \draw (3.5,1) -- (5,0); \draw (4.5,1) -- (4,0); \draw (4.5,1) --(5,0);

\draw (3,0) node[circle,fill,inner sep=2pt, color=blue](a){} -- (4,0) node[circle,fill,inner sep=2pt, color=blue](a){} -- (5,0) node[circle,fill,inner sep=2pt, color=blue](a){};

\end{tikzpicture}
\caption{Illustration of some right-angled Artin groups with join defining graphs and their generalised Bestvina-Brady subgroups.}
\label{fig2}
\end{figure}

\begin{rem}
The above theorem gives us a necessary and sufficient condition for a generalized Bestvina-Brady subgroup to be finitely generated. However, finding a finite generating set of a generalized Bestvina-Brady subgroup and computing its distortion in the right-angled Artin group is quite difficult. Therefore, we only focus on a subcollection of generalised Bestvina-Brady subgroups. 
\end{rem}

\begin{defn}
Let $\Gamma$ be a simplicial graph and $A_{\Gamma}$ the associated right-angled Artin group. A generalized Bestvina-Brady subgroup $H_\Phi$ is \emph{strong} if each basis subgraph of $\Gamma$ with respect to $\Phi$ is connected and strongly dominating. A generalized Bestvina-Brady subgroup $H_\Phi$ \emph{special} if each basis subgraph of $\Gamma$ with respect to $\Phi$ is connected and specially dominating.
\end{defn}

\begin{rem}
A special generalized Bestvina-Brady subgroup is obviously strong. The collection of strong generalized Bestvina-Brady subgroups of a right-angled Artin group includes its real Bestvina-Brady subgroup. If a right-angled Artin group $A_{\Gamma}$ contains a special generalized Bestvina-Brady subgroup, then $\Gamma$ must be a join. 
\end{rem}

\begin{exmp}
Let $\Gamma$ be one of the graphs in Figure \ref{fig1} or Figure \ref{fig2}. Let $\Phi$ be a group homomorphism from $A_\Gamma$ to $\Z^2$ that maps each red (top) vertex to $e_1=(1,0)$ and each blue (bottom) vertex to $e_2=(0,1)$. Let $H_\Phi$ be the generalised Bestvina-Brady subgroup with respect to $\Phi$. Then all $H_\Phi$ are finitely generated except for the case $\Gamma$ is the graph on the right of Figure \ref{fig1}. If $\Gamma$ is one of the graphs on Figure \ref{fig2} or the graph on the left of Figure \ref{fig1}, then the generalised Bestvina-Brady subgroup $H_\Phi$ is strong. If $\Gamma$ is the graph in the middle of Figure \ref{fig1}, then the generalised Bestvina-Brady subgroup $H_\Phi$ is not strong. If $\Gamma$ is the graph on the right of Figure \ref{fig2}, then the generalised Bestvina-Brady subgroup $H_\Phi$ is special.
\end{exmp}

We recall that each $S_i$ is the set of vertices of $\Gamma$ that are mapped to $e_i$ via $\Phi$ and each basis subgraph $\Gamma_i$ is the induced subgraph of $\Gamma$ with vertex set $S_i$. Let $T_i$ be the set of group elements of the form $uv^{-1}$, where $u$ and $v$ are two adjacent vertices in $\Gamma_i$ and $T$ the union of all sets $T_i$. Let $H_{T}$ be a subgroup of $A_\Gamma$ generated by $T$. For each $i$ in $\{1,2,\cdots,d\}$ let $p_i$ be a group homomorphism from $\Z^d$ to $\Z$ that maps $e_i$ to 1 and each $e_j$ to 0 for $j\neq i$. Let $K=\diam{\Gamma}+2$. Before we compute a upper bound of the distortion of a strong generalized Bestvina-Brady subgroup $H_\Phi$ in $A_{\Gamma}$, we need to prove several lemmas. For the following two lemmas, we assume that each basis subgraph $\Gamma_i$ is connected and strongly dominating. 

\begin{lem}
\label{lem1}
Let $n$ be an arbitrary integer and $s_1$, $s_2$ two elements in some $S_i$. Then the group elements $s_1^ns_2s_1^{-n-1}$ and $s_1^ns^{-1}_2s_1^{-n+1}$ both belong to $H_{T}$ and their length with respect to $T$ are both bounded above by $2K\bigl(\abs{n}+1\bigr)$.
\end{lem}
\begin{proof}
We can write $s_1^ns_2s_1^{-n-1}=(s_1^ns_2^{-n})(s_2^{n+1}s_1^{-n-1})$. Since $\Gamma_i$ is connected, there is a sequence $s_1=u_1, u_2, \cdots, u_{\ell}=s_2$ of elements in $S_i$, where $\ell \leq \diam{\Gamma} +1$ and $u_j$, $u_{j+1}$ commutes for each $i$ in $\{1, 2,\cdots, \ell -1\}$. Therefore, \[s_1^ns_2^{-n}=(u_1^nu_2^{-n})(u_2^nu_3^{-n})\cdots (u_{\ell-1}^nu_\ell^{-n})=(u_1u_2^{-1})^n(u_2u_3^{-1})^n\cdots (u_{\ell-1}u_\ell^{-1})^n.\]
This implies that the element $s_1^ns_2^{-n}$ belongs to $H_{T}$ and $\abs{s_1^ns_2^{-n}}_T\leq \ell \abs{n}\leq K\abs{n}$. Similarly, the element $s_2^{n+1}s_1^{-n-1}$ also belongs to $H_{T}$ and $\abs{s_2^{n+1}s_1^{-n-1}}_T\leq K\bigl(\abs{n}+1\bigr)$. Therefore, the group element $s_1^ns_2s_1^{-n-1}$ belongs to $H_{T}$ and $\abs{s_1^ns_2s_1^{-n-1}}_T\leq 2K\bigl(\abs{n}+1\bigr)$. Similarly, we can prove $s_1^ns^{-1}_2s_1^{-n+1}$ belongs to $H_{T}$ and $\abs{s_1^ns^{-1}_2s_1^{-n+1}}_T\leq 2K\bigl(\abs{n}+1\bigr)$
\end{proof}

\begin{lem}
\label{lem2}
Let $w$ be a word in $S_i$ and $s$ a generator in $S_i$ for some $i$. Let $m$ be an arbitrary integer and $n=\bigl(p_i\circ\Phi\bigr)(w)$. Then the word $s^m w s^{-m-n}$ represents an element in $H_{T}$ and $\abs{s^m w s^{-m-n}}_T\leq 2K\bigl(\abs{m}\ell(w)+\ell^2(w)\bigr)$. If $s$ commutes with all generators in the word $w$, then $\abs{s^m w s^{-m-n}}_T\leq \ell(w)$. 
\end{lem}

\begin{proof}
Assume $w=s_1^{n_1}s_2^{n_2}\cdots s_\ell^{n_\ell}$, where $\ell$ is the length of $w$, each $s_j$ is element in $S_i$, $n_j=1 \text{ or } -1$, and $n_1+n_2+\cdots+n_\ell=n$. We can write

\[s^m w s^{-m-n} =(s^m s_1^{n_1} s^{-m-n_1})(s^{m+n_1} s_2^{n_2} s^{-m-n_1-n_2})\cdots(s^{m+n_1+n_2+\cdots+n_{\ell -1}} s_\ell^{n_\ell} s^{-m-n_1-n_2-\cdots-n_{\ell}}).\]

By Lemma \ref{lem1}, the word $s^m w s^{-m-n}$ represents an element in $H_{T}$ and 
\begin{align*}\abs{s^m w s^{-m-n}}_T &\leq 2K\bigl(\abs{m}+1\bigr)+2K\bigl(\abs{m+n_1}+1\bigr)+\cdots +2K\bigl(\abs{m+n_1+n_2+\cdots+n_{\ell -1}}+1\bigr)\\&\leq 2K\ell\bigl(\abs{m}+\ell\bigr)\leq 2K\ell(w)\bigl(\abs{m}+\ell(w)\bigr)\leq 2K\bigl(\abs{m}\ell(w)+\ell^2(w)\bigr).\end{align*}

If $s$ commutes with all generators in the word $w$, then 
\begin{align*}
s^m w s^{-m-n}&\equiv_{A_{\Gamma}}s_1^{n_1}s_2^{n_2}\cdots s_\ell^{n_\ell} s^{-n}\\&\equiv_{A_{\Gamma}} (s_1s^{-1})^{n_1}(s_2s^{-1})^{n_2}\cdots (s_\ell s^{-1})^{n_\ell}.
\end{align*}

Therefore, $\abs{s^m w s^{-m-n}}_T\leq \ell(w)$.
\end{proof}

\begin{prop}
\label{pp}
Let $\Gamma$ be a finite simplicial graph. Let $H_\Phi$ be a generalized Bestvina-Brady subgroup of $A_{\Gamma}$. If the generalized Bestvina-Brady subgroup $H_\Phi$ is strong, then the distortion of $H_\Phi$ in $A_{\Gamma}$ is at most quadratic. Moreover, if the generalized Bestvina-Brady subgroup $H_\Phi$ is special, then the distortion of $H_\Phi$ in $A_{\Gamma}$ is linear.
\end{prop}

\begin{proof}
For each positive integer $n$ let $h$ be an arbitrary element in $H_\Phi$ such that $\abs{h}_S\leq n$. There is a word $w=w_1w_2\cdots w_k$ in $S$ that represent $h$ with the following properties:
\begin{enumerate}
\item The length of $w$ is less than or equal $n$.
\item For each $j$ in $\{1,2,\cdots, k\}$ the subword $w_j$ can be decomposed as $w_j=u_j^{(1)}u_j^{(2)}\cdots u_j^{(d)}$, where each $u_j^{(i)}$ is an empty word or a word in $S_i$.
\end{enumerate}

We will modify $w$ to obtain a new word $\bar{w}=\bar{w}_1\bar{w}_2\cdots \bar{w}_k$ that also represents $h$ with the following properties:

For each $j$ in $\{1,2,\cdots, k\}$ the subword $\bar{w}_j$ can be decomposed as $\bar{w}_j=\bar{u}_j^{(1)}\bar{u}_j^{(2)}\cdots \bar{u}_j^{(d)}$, where each $\bar{u}_j^{(i)}$ is an empty word or a word in $S_i$ such that $\bar{u}_j^{(i)}$ represent an element in $H_T$ and $\abs{\bar{u}_j^{(i)}}_T\leq 2K\bigl(n\ell(u_j^{(i)})+\ell^2(u_j^{(i)})\bigr)$.

Since $\Gamma_1$ is strongly dominating, then there is a vertex $s_1$ in $S_1$ that commutes with all elements in $S-S_1$. For each $j$ in $\{1,2,\cdots, k\}$ each $u_j^{(1)}$ is an empty word or a word in $S_1$. Let $n_j=p_1\circ\Phi(u_j^{(1)})$. Then $\abs{n_j}\leq \ell(u_j^{(1)})$ and $n_1+n_2+\cdots+n_k=0$. For each $j$ in $\{0, 1,2,\cdots, k\}$ we let $m_0=0$ and $m_j=n_1+n_2+\cdots+n_j$. We observe that $m_j=m_{j-1}+n_j$ and $m_k=0$. Let $\bar{u}_j^{(1)}=s_1^{m_{j-1}}u_j^{(1)}s_1^{-m_{j}}$. By Lemma \ref{lem2}, $\bar{u}_j^{(1)}$ represent an element in $H_T$ and $\abs{\bar{u}_j^{(1)}}_T\leq 2K\bigl(\abs{m_{j-1}}\ell(u_j^{(1)})+\ell^2(u_j^{(1)})\bigr)\leq 2K\bigl(n\ell(u_j^{(1)})+\ell^2(u_j^{(1)})\bigr)$. 

We note that the vertex $s_1$ in $S_1$ that commutes with all elements in $S-S_1$. Therefore, 
\begin{align*}
w_j &\equiv_{A_\Gamma}(s_1^{-m_{j-1}}s_1^{m_{j-1}})u_j^{(1)}(s_1^{-m_j}s_1^{m_j})u_j^{(2)}\cdots u_j^{(d)}\\&\equiv_{A_\Gamma}s_1^{-m_{j-1}}(s_1^{m_{j-1}}u_j^{(1)}s_1^{-m_j})(s_1^{m_j}u_j^{(2)}\cdots u_j^{(d)})\\&\equiv_{A_\Gamma}s_1^{-m_{j-1}}(\bar{u}_j^{(1)}u_j^{(2)}\cdots u_j^{(d)})s_1^{m_j}.
\end{align*}

This implies that $h$ can be represented by the word
\[(\bar{u}_1^{(1)}u_1^{(2)}\cdots u_1^{(d)})(\bar{u}_2^{(1)}u_2^{(2)}\cdots u_2^{(d)})\cdots (\bar{u}_k^{(1)}u_k^{(2)}\cdots u_k^{(d)}).\]

By repeating the same process, the element $h$ can be represented by a word $\bar{w}=\bar{w}_1\bar{w}_2\cdots \bar{w}_k$ with the following property:

For each $j$ in $\{1,2,\cdots, k\}$ the subword $\bar{w}_j$ can be decomposed as $\bar{w}_j=\bar{u}_j^{(1)}\bar{u}_j^{(2)}\cdots \bar{u}_j^{(d)}$, where each $\bar{u}_j^{(i)}$ is an empty word or a word in $S_i$ such that $\bar{u}_j^{(i)}$ represent an element in $H_T$ and $\abs{\bar{u}_j^{(i)}}_T\leq 2K\bigl(n\ell(u_j^{(i)})+\ell^2(u_j^{(i)})\bigr)$. Therefore, the word $\bar{w}_j$ represents elements in $H_T$ and $\abs{\bar{w}_j}_T$ is bounded above by $2K\bigl(n\sum_{i=1}^{d}\ell(u_j^{(i)})+\sum_{i=1}^{d}\ell^2(u_j^{(i)}) \bigr)$. This implies that $\abs{\bar{w}_j}_T$ is bounded above by $2K\bigl(n\ell(w_j)+\ell^2(w_j) \bigr)$. Thus, $h$ is an element in $H_T$ and $\abs{h}_T$ is bounded above by $2K\bigl(n\sum_{j=1}^{k}\ell(w_j)+\sum_{j=1}^{k}\ell^2(w_j)\bigr)$. Since $\sum_{j=1}^{k} \ell(w_j)\leq n$, $\abs{h}_T$ is bounded above by $4Kn^2$. Therefore, the distortion of $H_\Phi$ in $A_{\Gamma}$ is at most quadratic.

We now assume that the generalized Bestvina-Brady subgroup $H_\Phi$ is special. Then for each basis subgraph $\Gamma_i$ of $\Gamma$ with respect to $\Phi$ there is a vertex $s_i$ that commutes with all vertices of $\Gamma$. Therefore, the length of the element in $H_T$ represented by the word $\bar{u}_j^{(i)}$ must be less than or equal to $\ell(u_j^{(i)})$ by Lemma \ref{lem2}. Therefore, the length of the element in $H_T$ represented by the word $\bar{w}_j$ is bounded above by $\ell(w_j)$. This implies that $\abs{h}_T$ is bounded above by $\ell(w)$. Therefore, the distortion of $H_\Phi$ in $A_{\Gamma}$ is linear.

\end{proof}

\begin{figure}
\begin{tikzpicture}

\draw (-2,2) node[circle,fill,inner sep=2pt, color=red, label=above:$x_1$](a){} -- (0,2) node[circle,fill,inner sep=2pt, color=red, label=above:$x_2$](a){} -- (2,2) node[circle,fill,inner sep=2pt, color=red, label=above:$x_3$](a){}; 

\draw (-2,2) -- (-2,0); \draw (-2,2) -- (0,0); \draw (0,2) -- (0,0); \draw (0,2) -- (2,0); \draw (2,2) -- (2,0);

\draw (-2,0) node[circle,fill,inner sep=2pt, color=blue, label=below:$x_4$](a){} -- (0,0) node[circle,fill,inner sep=2pt, color=blue, label=below:$x_5$](a){} -- (2,0) node[circle,fill,inner sep=2pt, color=blue, label=below:$x_6$](a){};

\end{tikzpicture}
\caption{Illustration of a right-angled Artin group and its generalised Bestvina-Brady subgroup which is not strong.}
\label{fig4}
\end{figure}

We now compute the lower bound on the distortion of a strong generalized Bestvina-Brady subgroup $H_\Phi$. 
The following proposition is deduced from Theorem \ref{th1}, Theorem \ref{thbc}, and Proposition \ref{pp}.

\begin{prop}
\label{prop1}
Let $\Gamma$ be a simplicial connected graph with at least 2 vertices. Let $H_{\Phi}$ be a strong generalized Bestvina-Brady subgroup. If the graph $\Gamma$ is not a join, then the distortion of $H_{\Phi}$ in $A_{\Gamma}$ is quadratic.
\end{prop}

\begin{rem}
In the above proposition, we see the equivalence between subgroup distortion of a strong generalized Bestvina-Brady subgroup $H_{\Phi}$ in $A_{\Gamma}$ and the divergence of group $A_\Gamma$. However, this equivalence does not always occur if $\Gamma$ is a join graph. In the rest of this section, we will investigate the distortion of strong generalized Bestvina-Brady subgroups.
\end{rem}

\begin{thm}
Let $\Gamma$ be a simplicial connected graph. Let $H_\Phi$ be a strong generalised Bestvina-Brady subgroup of $A_\Gamma$ with basis subgraphs $\{\Gamma_i\}_{1\leq i \leq d}$. Then, the distortion of $H_\Phi$ in $A_\Gamma$ is quadratic if there is a non-empty subset $I$ of $\{1,2, \cdots, d\}$ such that the subgraph $\Gamma'$ generated by $\{\Gamma_i\}_{i\in I}$ is not a join. Moreover, the distortion of $H_\Phi$ in $A_\Gamma$ is linear if one of the following conditions holds:
\begin{enumerate}
\item The generalised Bestvina-Brady subgroup $H_\Phi$ is special.
\item Each basis subgraphs in $\{\Gamma_i\}_{1\leq i \leq d}$ is a join and each pair of basis subgraphs in $\{\Gamma_i\}_{1\leq i \leq d}$ commutes. 
\end{enumerate}
\end{thm}

\begin{proof}
We assume that there is a non-empty subset $I$ of $\{1,2, \cdots, d\}$ such that the subgraph $\Gamma'$ generated by $\{\Gamma_i\}_{i\in I}$ is not a join. We will prove that the distortion of $H_\Phi$ in $A_\Gamma$ is quadratic. We recall that each $S_i$ is the set of vertices of $\Gamma$ that is mapped to $e_i=(0,0,\cdots, 1^{i th}, \cdots, 0)$ in $\Z^d$ via $\Phi$ and each basis subgraph $\Gamma_i$ is the induced subgraph of $\Gamma$ with vertex set $S_i$. Each $T_i$ is the set of group elements of the form $uv^{-1}$, where $u$ and $v$ are two adjacent vertices in $\Gamma_i$ and $T$ is the union of all sets $T_i$. Let $S'$ be the union of all sets in $\{S_i\}_{i\in I}$. Then $\Gamma'$ is the induced subgraph with vertex set $S'$. 

Let $d'$ be the number of elements in $I$. We can consider $\Z^{d'}$ as the subgroup of $Z^d$ generated by $\{e_i\}_{i\in I}$. Then there is a group homomorphism $\Phi'$ from $\Gamma'$ to $\Z^{d'}$ such that $\Phi'=\Phi_{|A_{\Gamma'}}$. Therefore, the generalised Bestvina-Brady subgroup $H_{\Phi'}$ of $A_{\Gamma'}$ is also the subgroup of $H_\Phi$ and $\{\Gamma_i\}_{i\in I}$ is the collection of basis subgraphs of $\Gamma$ with respect to $\Phi'$. Since each graph $\Gamma_i$ is connected and strongly dominating in $\Gamma$, each graph in $\{\Gamma_i\}_{i\in I}$ is also connected and strongly dominating in $\Gamma'$. Therefore, the generalised Bestvina-Brady subgroup $H_{\Phi'}$ of $A_{\Gamma'}$ is strong. Also, the subgraph $\Gamma'$ is not a join. Therefore, the distortion of $H_{\Phi'}$ in $A_{\Gamma'}$ is quadratic by Proposition \ref{prop1}. Also, the distortion of $A_{\Gamma'}$ in $A_{\Gamma}$ is linear. Therefore, the distortion of $H_{\Phi'}$ in $A_{\Gamma}$ is quadratic. 

We now prove that the distortion of $H_{\Phi'}$ in $H_{\Phi}$ is linear. Let $T'$ be the union of all sets in $\{T_i\}_{i\in I}$. Then $T'$ is a finite generating set for $H_\Phi'$. There is a group homomorphism $f$ from $A_{\Gamma}$ to $A_{\Gamma'}$ that maps each vertex of $\Gamma-\Gamma'$ to the identity and each vertex of $\Gamma'$ to itself. Therefore, the restriction of $f$ on $A_{\Gamma'}$ is the identity map. Moreover, the maps $f$ maps each element in $T'$ to itself and element in $T-T'$ to the identity. This fact implies that the distortion of $H_{\Phi'}$ in $H_{\Phi}$ is linear. Therefore, the distortion of $H_{\Phi}$ in $A_{\Gamma}$ is least quadratic. By Proposition \ref{pp}, the distortion of $H_{\Phi}$ in $A_{\Gamma}$ is exactly quadratic. 

If $H_\Phi$ is a special generalised Bestvina-Brady subgroup, then the distortion of $H_{\Phi}$ in $A_{\Gamma}$ is linear by Proposition \ref{pp}. We now assume that each basis subgraphs in $\{\Gamma_i\}_{1\leq i \leq d}$ is a join and each pair of basis subgraphs in $\{\Gamma_i\}_{1\leq i \leq d}$ commutes. In this case, the group $A_{\Gamma}$ can be written as the direct product $A_{\Gamma_1}\times A_{\Gamma_2}\times \cdots \times A_{\Gamma_d}$ and the subgroup $H_\Phi$ can be written as the direct product $H_{\Gamma_1}\times H_{\Gamma_2}\times \cdots \times H_{\Gamma_d}$, where each $H_{\Gamma_i}$ is the Bestvina-Brady subgroup of $A_{\Gamma_i}$. Since each graph $\Gamma_i$ is a join, the distortion of $H_{\Gamma_i}$ in $A_{\Gamma_i}$ is linear (see \cite{Tran}). Therefore, the distortion of $H_{\Phi}$ in $A_{\Gamma}$ is linear. 
\end{proof}

\begin{rem}
By the above theorem, we observe that the group divergence and the subgroup distortion can be equivalent for some cases (see the graph on the left of Figure \ref{fig1} or the graph on the right of Figure \ref{fig2} for example). However, the group divergence and the subgroup distortion are not necessarily equivalent (see the graph on the left of Figure \ref{fig2} for example). 
\end{rem}






We now compute the distortion of one generalised Bestvina-Brady subgroup which is not strong. More precisely, we compute the distortion of the generalised Bestvina subgroup illustrated by the middle graph of Figure \ref{fig1} (we call graph $\Gamma$). We label each vertex of this graph as in Figure \ref{fig4}. Obviously, the graph $\Gamma$ is not a join. Therefore, the distortion of the subgroup $H_\Phi$ is at least quadratic by Theorem \ref{th1} and Theorem \ref{thbc}. We now use Theorem \ref{theo1} to prove the upper bound for the distortion of $H_\Phi$. Let $G_1$ be the subgroup generated by $\{x_1,x_4,x_5\}$, $G_2$ the subgroup generated by $\{x_1,x_2,x_5\}$, $G_3$ the subgroup generated by $\{x_2,x_5,x_6\}$, and $G_4$ the subgroup generated by $\{x_2,x_3,x_6\}$. Then the collection $\mathcal{A}=\{G_1, G_2, G_3, G_4\}$ and the normal subgroup $H_\Phi$ satisfy the hypothesis of Theorem \ref{theo1}. It is not hard to see that the distortion of each subgroup $G_i \cap H_\Phi$ is undistorted in $G_i$. Therefore, the distortion of the subgroup $H_\Phi$ in $A_\Gamma$ is at most quadratic. Thus, the distortion of the subgroup $H_\Phi$ in $A_\Gamma$ is exactly quadratic.

\begin{conj}
Let $\Gamma$ be a simplicial graph and $H_\Phi$ some finitely generated generalised Bestvina-Brady subgroup of $A_\Gamma$. Then the distortion of $H_\Phi$ in $A_\Gamma$ is linear or quadratic.
\end{conj}

\bibliographystyle{alpha}
\bibliography{Tran}

\def\cprime{$'$}
\begin{thebibliography}{BDM09}

\bibitem[BB97]{MR1465330}
Mladen Bestvina and Noel Brady.
\newblock Morse theory and finiteness properties of groups.
\newblock {\em Invent. Math.}, 129(3):445--470, 1997.

\bibitem[BB06]{MR2252898}
Josh Barnard and Noel Brady.
\newblock Distortion of surface groups in {CAT}(0) free-by-cyclic groups.
\newblock {\em Geom. Dedicata}, 120:119--139, 2006.

\bibitem[BC12]{MR2874959}
Jason Behrstock and Ruth Charney.
\newblock Divergence and quasimorphisms of right-angled {A}rtin groups.
\newblock {\em Math. Ann.}, 352(2):339--356, 2012.

\bibitem[BD14]{MR3421592}
Jason Behrstock and Cornelia Dru{\c{t}}u.
\newblock Divergence, thick groups, and short conjugators.
\newblock {\em Illinois J. Math.}, 58(4):939--980, 2014.

\bibitem[BDM09]{MR2501302}
Jason Behrstock, Cornelia Dru{\c{t}}u, and Lee Mosher.
\newblock Thick metric spaces, relative hyperbolicity, and quasi-isometric
  rigidity.
\newblock {\em Math. Ann.}, 344(3):543--595, 2009.

\bibitem[Ger94]{MR1254309}
S.~M. Gersten.
\newblock Quadratic divergence of geodesics in {${\rm CAT}(0)$} spaces.
\newblock {\em Geom. Funct. Anal.}, 4(1):37--51, 1994.

\bibitem[KS70]{MR0260879}
A.~Karrass and D.~Solitar.
\newblock The subgroups of a free product of two groups with an amalgamated
  subgroup.
\newblock {\em Trans. Amer. Math. Soc.}, 150:227--255, 1970.

\bibitem[Mac13]{MR3032700}
Nata{\v{s}}a Macura.
\newblock C{AT}(0) spaces with polynomial divergence of geodesics.
\newblock {\em Geom. Dedicata}, 163:361--378, 2013.

\bibitem[MV95]{MR1337468}
John Meier and Leonard VanWyk.
\newblock The {B}ieri-{N}eumann-{S}trebel invariants for graph groups.
\newblock {\em Proc. London Math. Soc. (3)}, 71(2):263--280, 1995.

\bibitem[PS09]{MR2466422}
Stefan Papadima and Alexander~I. Suciu.
\newblock Toric complexes and {A}rtin kernels.
\newblock {\em Adv. Math.}, 220(2):441--477, 2009.

\bibitem[Sis]{Sisto}
Alessandro Sisto.
\newblock On metric relative hyperbolicity.
\newblock Preprint. arXiv:1210.8081.

\bibitem[Tra17]{Tran}
Hung~Cong Tran.
\newblock Geometric embedding properties of {B}estvina-{B}rady subgroups.
\newblock {\em Algebr. Geom. Topol.}, 17(4):2499--2510, 2017.

\end{thebibliography}
\end{document}